\documentclass[12pt]{amsart}
\usepackage{amsfonts,amsthm,latexsym,amsmath,amssymb,amscd,amsmath, mathrsfs, url}
\newcounter{count}

\numberwithin{count}{section}
\newtheorem{Lemma}[count]{Lemma}

\newtheorem{Theorem}[count]{Theorem}
\newtheorem{Conjecture}[count]{Conjecture}

\begin{document}

\author[T. H.~Nguyen]{Thu Hien Nguyen}

\address{Partially supported by the Akhiexer Foundation. Department of Mathematics \& Computer Sciences, V. N. Karazin Kharkiv National University,
4 Svobody Sq., Kharkiv, 61022, Ukraine}
\email{nguyen.hisha@gmail.com }

\title[Entire functions of the Laguerre-P\'olya class]
{On the conditions for a special entire function relative to the
partial theta-function and the Euler function to belong to the 
Laguerre-P\'olya class }

\begin{abstract}
In this paper, we discuss the conditions for the function $F_a(z) = \sum_{k=0}^\infty \frac{z^k}{(a+1)(a^2+1) 
\ldots (a^k+1)}, a>1,$ to belong to the Laguerre-P\'olya class, or to have only real zeros.

\end{abstract}



\maketitle

\section{Introduction}

{\bf Definition 1}.  A real entire function $f$ is said to be in the {\it
Laguerre-P\'olya class}, written $f \in \mathcal{L-P}$, if it can
be expressed in the form
\begin{equation}
\label{lpc}
 f(x) = c x^n e^{-\alpha x^2+\beta x}\prod_{k=1}^\infty
\left(1-\frac {x}{x_k} \right)e^{xx_k^{-1}},
\end{equation}
where $c, \alpha, \beta, x_k \in  \mathbb{R}$, $x_k\ne 0$,  $\alpha \ge 0$,
$n$ is a nonnegative integer and $\sum_{k=1}^\infty x_k^{-2} < \infty$. As usual, the product on the right-hand side can be
finite or empty (in the latter case the product equals 1).

This class is essential in the theory of entire functions due to the fact that  the polynomials 
with only real zeros converge locally uniformly to these and only these 
functions. The following  prominent theorem 
states an even  stronger fact. 

{\bf Theorem A} (E.Laguerre and G.P\'{o}lya, see, for example,
\cite[p. 42-46]{HW}). {\it   

(i) Let $(P_n)_{n=1}^{\infty},\  P_n(0)=1, $ be a sequence
of complex polynomials having only real zeros which  converges uniformly in
the circle $|z|\le A, A > 0.$ Then this sequence converges
locally uniformly to an entire function
 from the $\mathcal{L-P}$ class.

(ii) For any $f \in \mathcal{L-P}$ there is a
sequence of complex polynomials with only real zeros which
converges locally uniformly to $f$.}

For various properties and characterizations of the
Laguerre-P\'olya class see 
\cite[p. 100]{pol}, \cite{polsch}  or  \cite[Kapitel II]{O}.

Note that for a real entire function (not identically zero) of order less than $2$ having 
only real zeros is  equivalent to belonging to the Laguerre-P\'olya class. The situation is 
different when an entire function is of order $2$. For example, the function 
$f_1(x)= e^{- x^2}$ belongs to the Laguerre-P\'olya class, but the function 
$f_2(x)= e^{x^2}$ does not.

Let  $f(z) = \sum_{k=0}^\infty a_k z^k$  be an entire function with positive coefficients. 
We define the quotients $p_n$ and $q_n$:

\begin{eqnarray}
\label{qqq} &  p_n=p_n (f):=\frac {a_{n-1}}{a_n},\ n\geq 1;
\\
\nonumber &
q_n=q_n(f):=\frac {p_n}{p_{n-1}}=\frac {a_{n-1}^2}{a_{n-2}a_n},\
n\geq 2.
\end{eqnarray}

The following formulas can be verified by straightforward calculation.

\begin{eqnarray}
\label{defq}
& a_n=\frac {a_0}{p_1 p_2 \ldots p_n},\ n\geq 1\ ; \\
\nonumber & a_n=\frac
{a_1}{q_2^{n-1} q_3^{n-2} \ldots q_{n-1}^2 q_n} \left(
\frac{a_1}{a_0} \right) ^{n-1},\ n\geq 2.
\end{eqnarray}

In 1926, J. I. Hutchinson found the following sufficient condition for an entire function 
with positive coefficients to have only real zeros.

{\bf Theorem B} (J. I. Hutchinson, \cite{hut}). { \it Let $f(z)=\sum_{k=0}^\infty a_k z^k$, $a_k > 0$ for all $k$. 
Then $q_n(f)\geq 4$, for all $n\geq 2,$  
if and only if the following two conditions are fulfilled:\\
(i) The zeros of $f(z)$ are all real, simple and negative, and \\
(ii) the zeros of any polynomial $\sum_{k=m}^n a_kz^k$, $m < n,$  formed 
by taking any number 
of consecutive terms of $f(z) $, are all real and non-positive.}

For some extensions of Hutchinson's results see,
for example, \cite[\S4]{cc1}. 

The entire function $g_a(z) =\sum _{j=0}^{\infty} z^j a^{-j^2}$, $a>1,$
a so-called \textit{partial theta-function},  was investigated in the paper \cite{klv}. 
Simple calculations show that $q_n(g_a)=a^2$ for all $n.$ 

The  survey \cite{War} by S.O. Warnaar contains the history of
investigation of the partial theta-function and its interesting properties.

In \cite{klv} it is shown that  for every $n\geq 2$ there exists a constant $c_n >1$ such 
that  $S_{n}(z,g_a):=\sum _{j=0}^{n} z^j a^{-j^2} \in \mathcal{L-P}$
$ \Leftrightarrow \ a^2 \geq c_n$.

{\bf Theorem C} (O. Katkova, T. Lobova, A. Vishnyakova, \cite{klv}).  {\it There exists a constant 
$q_\infty $ $(q_\infty\approx 3{.}23363666 \ldots ) $ such that:
\begin{enumerate}
\item
$g_a(z) \in \mathcal{L-P} \Leftrightarrow \ a^2\geq q_\infty ;$
\item
$g_a(z) \in \mathcal{L-P} \Leftrightarrow \ $  there exists $x_0 \in (- a^3, -a)$ such that $ \  g_a(x_0) \leq 0;$
\item
for a given $n\geq 2$ we have $S_{n}(z,g_a) \in \mathcal{L-P}$ $ \ \Leftrightarrow \ $
there exists $x_n \in (- a^3, -a)$ such that $ \ S_n(x_n,g_a) \leq 0;$
\item
$ 4 = c_2 > c_4 > c_6 > \ldots $  and    $\lim_{n\to\infty} c_{2n} = q_\infty ;$
\item
$ 3= c_3 < c_5 < c_7 < \ldots $  and    $\lim_{n\to\infty} c_{2n+1} = q_\infty .$
\end{enumerate}}

There is a series of works by V.P. Kostov dedicated to the interesting properties of zeros of the 
partial theta-function and its  derivative (see \cite{kos0}, \cite{kos1}, \cite{kos2}, \cite{kos3}, \cite{kos03},
\cite{kos04}, \cite{kos4}, \cite{kos5} and \cite{kos5.0}). For example, in  \cite{kos1}, V.P.~Kostov studied 
the so-called spectrum of the partial theta function, i.e. the set of values of 
$a>1$ for which the function  $g_a $ has a multiple real zero. 

A wonderful paper \cite{kosshap} among the other results explains the role 
of the constant $q_\infty $  in the study of the set of entire functions with positive 
coefficients having all Taylor truncations with only real zeros. 

{\bf Theorem D} (V.P. Kostov, B. Shapiro, \cite{kosshap}).
{\it Let $f(z) = \sum_{k=0}^\infty a_k z^k$ be an entire function 
with positive coefficients and $S_n(z) = \sum_{j=0}^n a_j z^j$ be its sections.
Suppose that there exists  $ N\in {\mathbb{N}},$ such that for all  $  n \geq N$ the sections
$S_n(z) = \sum_{j=0}^n a_j z^j$ belong to the Laguerre-P\'olya class.
Then $\lim\inf_{n\to \infty} q_n(f) \geq q_\infty$.
}

In \cite{klv1}, some  entire functions  with a convergent sequence of second quotients 
of coefficients  are investigated. The main question of \cite{klv1} is whether a function 
and its Taylor sections belong to the Laguerre-P\'olya class. In \cite{BohVish} and 
\cite{Boh}, some important special functions with increasing sequence of second quotients 
of Taylor coefficients  are studied. 

In \cite{ngthv1} and \cite{ngthv2}, the sufficient and necessary conditions were found for 
some entire functions of order zero to belong to the Laguerre-P\'olya class.

We  have studied the entire functions with positive Taylor coefficients such that 
$q_n(f)$ are decreasing in $n.$

{\bf Theorem  E} (T. H. Nguyen, A. Vishnyakova, \cite{ngthv1}).
{\it Let $f(z)=\sum_{k=0}^\infty a_k z^k $, $a_k > 0$ for all $k$, be an 
entire function.  Suppose that $q_n(f)$ are decreasing in $n$, i.e.  $q_2 \geq q_3 \geq 
q_4 \geq \ldots, $  and  $\lim\limits_{n\to \infty} q_n(f) = b \geq q_\infty$. 
Then all the zeros of $f$ are  real and negative, in other words  $f \in \mathcal{L-P}$.}

It is easy to see that, if only the estimation of $q_n(f)$ from below is given
and the assumption of monotonicity is omitted, then the constant $4$ in $q_n (f) \geq 4 $ 
is the smallest possible to conclude that $f \in \mathcal{L-P}$. 

We have also investigated the case when $q_n(f)$ are increasing in $n$ and have obtained the following theorem.

{\bf Theorem F} (T. H. Nguyen, A. Vishnyakova, \cite{ngthv2}).
{\it Let $f(z)=\sum_{k=0}^\infty a_k z^k $, $a_k > 0$ for all $k,$  be an 
entire function.  Suppose that the quotients $q_n(f)$ are increasing in $n$, 
and  $\lim\limits_{n\to \infty} q_n(f) = c < q_\infty$. 
Then the function $f$ does not belong to the  Laguerre-P\'olya class.}

A well known function 
$$h_a(z) = \sum \limits_{k=0}^\infty \frac{z^k}{(a^k -1)(a^{k-1} - 1)\cdot \ldots \cdot (a-1)} = 
\prod \limits_{k=1}^\infty \left(1 + \frac{z}{a^k} \right),\  a > 1,$$ 
has only real negative zeros. 

In this paper, we study the following function
$$F_a(z) = \sum_{k=0}^\infty \frac{z^k}{(a^k+1)(a^{k-1}+1)\cdot \ldots \cdot (a+1)},$$ 
and we want to find out, for which $a>1$ this function belongs to the Laguerre-P\'olya class. 
This problem was posed in the problem list of the workshop ``Stability, hyperbolicity, and zero 
localization of functions'' (American Institute of Mathematics, Palo Alto, California, 2011, see 
\cite[Problem 8.2]{AIM}). We have  $q_n(F_a) = \frac{a_{n-1}^2}{a_{n-2}a_n} = \frac{a^n +1}{a^{n-1}+1}$ 
and $\lim\limits_{n\to \infty} q_n(F_a)= a.$ It is easy to see that the second quotients of $F_a$ 
are increasing in $n$:
$\frac{a^n+1}{a^{n-1}+1} < \frac{a^{n+1}+1}{a^n+1}$ is equivalent to
$(a^n+1)^2 < (a^{n-1}+1)(a^{n+1}+1),$ or $2a < a^2 + 1.$
Thus, $q_n(F_a) < q_{n+1}(F_a)$ for all $n \geq 2.$ 

We are going to prove the two following theorems.

\begin{Theorem}
\label{th:mthm1}
The entire function $F_a(z) = \sum_{k=0}^\infty \frac{z^k}{(a^k+1)(a^{k-1}+1)\cdot \ldots \cdot (a+1)},$ 
$a>1,$ belongs to the Laguerre-P\'olya class if and only if there exists $z_0 \in (- (a^2+1), - (a+1))$ such that 
$F_a(z_0) \leq 0$.
\end{Theorem}

In order to sharpen this result, we will prove the following theorem.

\begin{Theorem}
\label{th:mthm2}

\begin{itemize}
\item[(i)]If $F_a(z) = \sum_{k=0}^\infty \frac{z^k}{(a^k+1)(a^{k-1}+1)\cdot \ldots \cdot (a+1)}, a>1,$ 
belongs to the Laguerre-P\'olya class, then $a \geq 3.90155;$\\
\item [(ii)]If $a \geq 3.91719 ,$ then $F_a(z) = 
\sum_{k=0}^\infty \frac{z^k}{(a^k+1)(a^{k-1}+1)\cdot \ldots \cdot (a+1)}, 
a>1,$ belongs to the Laguerre-P\'olya class.
\end{itemize}
\end{Theorem}

The question, for which $a >1$ does the entire function $F_a(z) = \sum_{k=0}^\infty \frac{z^k}{(a^k+1)(a^{k-1}+1)\cdot \ldots \cdot (a+1)}$ belong to the Laguerre-P\'olya class, aroused our interest. The following statement has not been proved due to some technical reasons. For now, we would like to leave it for the reader as an open problem.
\begin{Conjecture}
Let $F_a(z) = \sum_{k=0}^\infty \frac{z^k}{(a^k+1)(a^{k-1}+1)\cdot \ldots \cdot (a+1)}, a>1,$ be an entire function. Suppose that for $a_1>1$, $F_{a_1}$ belongs to the 
Laguerre-P\'olya class if and only if there exists $z_1 \in (-(a_1^2+1), -(a_1+1))$ such that $F_{a_1}(z_1) \leq 0$.
Then, for any $a_2 > a_1$, there exists $z_2 \in (-(a_2^2+1), -(a_2+1))$ such that $F_{a_2}(z_2) \leq 0.$
\end{Conjecture}

\section{Proof of Theorem \ref{th:mthm1}}

The following lemma from \cite{ngthv2}  shows that 
for $q_2(F_a)  <3$ we have  $F_a \notin \mathcal{L-P}.$

\begin{Lemma} (\cite[Lemma 2.1]{ngthv2}).
\label{th:lm0}
Let $\varphi(z) = \sum_{k=0}^\infty (-1)^k a_k z^k$ be an entire function, 
$a_k > 0$ for all $k,$ $a_0 = a_1 = 1$,  and $q_n =q_n(\varphi)$ are increasing in $n$, i.e.  
$q_2 \leq q_3 \leq q_4 \leq \ldots . $  
If $\varphi \in \mathcal{L-P},$ then  $q_2(f) \geq 3$.
\end{Lemma}

For the reader's convenience we give the proof.

\begin{proof}
Suppose that  $ \varphi \in \mathcal{L-P},$ and denote by  $0 < z_1 \leq z_2 \leq z_3 \leq  \ldots $  
the real roots of $\varphi $. 
We observe that
$$0 \leq \sum_{k=1}^\infty \frac{1}{z_k^2} = \left( \sum_{k=1}^\infty \frac{1}{z_k}  \right)^2 - 
2 \sum_{1\leq i< j < \infty} \frac{1}{z_i z_j} = \left(\frac{a_1}{a_0}\right)^2 - 2\frac{a_2}{a_0},$$
whence  $q_2 \geq 2.  $

According to the Cauchy-Bunyakovsky-Schwarz inequality, we obtain
$$(\frac{1}{z_1} + \frac{1}{z_2} + ...)(\frac{1}{z_1^3} + \frac{1}{z_2^3} + 
\ldots) \geq (\frac{1}{z_1^2 }+ \frac{1}{z_2^2} + \ldots)^2.$$

By Vieta's formulas, we have $ \sigma_1:= \sum_{k=1}^\infty \frac{1}{z_k} =  \frac{a_1}{a_0},$ 
$\sigma_2 = \sum_{1<i<j<\infty} \frac{1}{z_iz_j}$ $= \frac{a_2}{a_0},$ and
$\sigma_3 = \sum_{1<i<j<k<\infty} \frac{1}{z_iz_jz_k}
=  \frac{a_3}{a_0}.$ 
Further, we need the following identities:
$\sum_{k=1}^\infty \frac{1}{z_1^2} = \sigma_1^2 - 2\sigma_2,$ and 
$\sum_{k=1}^\infty \frac{1}{z_1^3} = \sigma_1^3 - 3\sigma_1\sigma_2 + 3\sigma_3.$ 
Consequently, we have 
$$\sigma_1 (\sigma_1^3 - 3\sigma_1\sigma_2 + 3\sigma_3) \geq (\sigma_1^2 - 2\sigma_2)^2,$$
or
$$\frac{a_1^2a_2}{a_0^3} + 3\frac{a_1a_3}{a_0^2} - 4\frac{a_2^2}{a_0^2} \geq 0.  $$

Since $a_0 = a_1 = 1$ and $a_2 = \frac{1}{q_2}, a_3 = \frac{1}{q_2^2q_3}$, we get: 
$$q_3(q_2 - 4) + 3 \geq 0.$$
Since we have the condition  $q_2 \leq q_3,$ supposing $q_2 < 4,$  we conclude that
$$q_2(q_2 - 4) + 3 \geq 0.$$
Therefore, we get that $q_2 \geq 3.$
\end{proof} 

So, if $F_a \in \mathcal{L-P},$ then  $q_2(F_a) \geq 3$.  If $q_2(F_a) \geq 4$, then for 
any $j$  $q_j(F_a) \geq 4$, so, according to the Hutchinson's theorem B, $F_a \in \mathcal{L-P}.$ 
It remains to consider the case $q_2(F_a) \in [3, 4).$

Thus, if $F_a(z) \in \mathcal{L-P}$, $q_2(F_a) = \frac{a^2+1}{a+1} \geq 3.$ It follows that 
$a \geq \frac{3+\sqrt{17}}{2} \approx 3.56155281 \ldots \geq q_\infty.$
Note that $\lim\limits_{n\to \infty} q_k(F_a) = a > q_\infty.$
If $q_2(F_a) < 4$, then $a < 2+ \sqrt{7} \approx 4.64575131\ldots.$ 
Accordingly, we look for $a \in (\frac{3+\sqrt{17}}{2}, 2+ \sqrt{7}).$

Further, we consider the entire function $f_a(z)= F_a (-z) =$ \\ $ \sum_{k=0}^\infty 
\frac{(-1)^kz^k}{(a^k +1)(a^{k-1}+1) \ldots (a+1)}$, and $q_k(f_a)$ are increasing in $k$, 
$\lim\limits_{n\to \infty} q_k(f_a) = a > q_\infty.$ Let $S_{n, a}(z) :=\sum_{k=0}^n 
\frac{(-1)^kz^k}{(a^k +1)(a^{k-1}+1) \ldots (a+1)}$ be the sections of $f_a$. 

We also need the Lemma below.

\begin{Lemma}
\label{th:lm1}
$\min_{|z| = a^2 +1}|S_{2, a}(z)| = 1.$ 
\end{Lemma}

\begin{proof}
In the proof we use the denotation $q_2$ instead of $q_2(f_a)$  ($q_2 = \frac{a^2+1}{a+1}$). 
By straightforward calculation, we have

$$|S_{2, a}((a^2 +1)e^{i\theta})|^2 = |1 - \frac{a^2+1}{a+1}e^{i\theta} + 
\frac{a^2+1}{a+1}e^{2i\theta}|^2=$$
$$(1 - q_2 \cos(\theta) +q_2 \cos(2\theta))^2   + (-q_2 \sin(\theta) + q_2 \sin(2\theta))^2 = $$
$$1+ 2q_2^2 - 2q_2(1+q_2)\cos(\theta) + 2q_2\cos(2\theta). $$

Set $t: = \cos(\theta), t \in [-1,1]$. It follows that $\cos(2\theta) = 2t^2-1$, and we get

$|S_ {2,a}((a^2 +1)e^{i\theta})|^2 = 4q_2t^2 - 2q_2(1+q_2)t + 1 - 2q_2 + 2q_2^2 =:\xi(t).$

As a result, we have obtained a quadratic expression, and we consider its discriminant. 
$D_\xi/4 = q_2^4 - 6q_2^3 + 9q_2^2 - 4q_2 = q_2(q_2 -1)^2(q_2-4).$ Under 
our assumptions, $q_2 < 4 $, so we can see that the discriminant is negative. 
Therefore, $\xi(t)$ has no zeros and $\xi(t)>0.$

The vertex point of the parabola is $t_v = \frac{1+q_2}{4} \geq 1$,  since 
$q_2 \geq 3$. So, $\min_{t\in[-1,1]}\xi(t) = \xi(1) = 1.$ Thus, 
$\min_{|z|=a^2+1}|S_{2, a}(z)|=1.$
\end{proof}

Now we want to get the estimation of the modulus of $R_{3, a} (z) :=  \sum_{k=3}^\infty 
\frac{(-1)^kz^k}{(a^k +1)(a^{k-1}+1) \ldots (a+1)}$ from above.

\begin{Lemma}
\label{th:lm2} 
$\max_{|z| = a^2 +1}|R_{3, a}(z)| \leq \frac{(a^2+1)^2}{(a+1)(a^3+1)} \cdot 
\frac{a^4+1}{a^4 - a^2}.$
\end{Lemma}

\begin{proof} We have

$$\max_{|z| = a^2 +1}|R_{3, a}(z)| \leq  \sum_{k=3}^\infty 
\left| \frac{(a^2+1)^k}{(a+1)(a^2+1) \ldots (a^k+1)} \right| = \frac{(a^2+1)^2}{(a+1)(a^3+1)} $$
$$ + \frac{(a^2+1)^3}{(a+1)(a^3+1)(a^4+1)}
+ \frac{(a^2+1)^5}{(a+1)(a^3+1)(a^4+1)(a^5+1)} + \ldots =$$
$$\frac{(a^2+1)^2}{(a+1)(a^3+1)} \left( 1 + \frac{a^2+1}{a^4+1} + 
\frac{(a^2+1)^2}{(a^4+1)(a^5+1)} + \ldots \right) $$
$$\leq \frac{(a^2+1)^2}{(a+1)(a^3+1)} \cdot \frac{1}{1 - \frac{a^2+1}{a^4+1}}= 
\frac{(a^2+1)^2}{(a+1)(a^3+1)} \cdot \frac{a^4+1}{a^4 - a^2}.$$
\end{proof}

Hence,  using the Lemmas above, we want to get the following: 
$$\max_{|z| = a^2 +1}|R_{3, a}(z)| \leq \frac{(a^2+1)^2}{(a+1)(a^3+1)} 
\cdot \frac{a^4+1}{a^4 - a^2} < \min_{|z| = a^2+1}|S_{2,a}(z)| = 1.$$
The inequality $$\frac{(a^2+1)^2}{(a+1)(a^3+1)} 
\cdot \frac{a^4+1}{a^4 - a^2} < 1$$ is equivalent to 
$$ a^7 - 3a^6 - a^4 - a^3 - 3a^2 - 1 > 0.$$  The latter is valid for $a \geq 3.16258 \ldots,$ 
so, under our assumptions that  $a \geq \frac{3+\sqrt{17}}{2} \approx 3.56155281 \ldots$, the inequality is fulfilled. Consequently, according to Rouch\'e's theorem, the functions $f_a$ and $S_{n, a}, n \geq 3,$
have the same number of zeros (counting multiplicities) inside the circle $\{ z: |z| < a^2 +1\} $  
as the polynomial $S_{2, a}.$

The discriminant of the polynomial $S_{2, a}$ under our assumption is negative:
$D = (a^2+1)^2 - 4(a+1)(a^2+1) = (a^2+1)(a^2 - 4a - 3) < 0.$
Thus, $S_{2,a}(z)$ has 2 complex conjugate zeros, and their modulus are equal to $\sqrt{(a+1)(a^2+1)} 
< a^2 +1$. Therefore, the polynomial $S_{2, a}$ has exactly two  zeros inside the circle $\{ z: |z| < a^2 +1\} $ 
for such values of $a$ that $q_2(f_a) \in [3, 4)$.

We have proved that the functions $f_a(z)$ and $S_{n, a}(z), n \geq 2,$
have exactly two zeros (counting multiplicities) inside the circle $\{ z: |z| < a^2 +1\} $ 
for such values of $a$ that $q_2(f_a) \in [3, 4)$. Thus, if $f_a \in \mathcal{L-P}$ (or 
$S_{n, a} \in \mathcal{L-P}$ for $n\geq 2$),   then 
these two zeros are real, and there  exists $z_0 \in (a+1, a^2+1)$ such that $f_a(z_0) \leq 0$.

Since $f_a(-x)$ and $ S_{n,a}(-x)$ have positive Taylor coefficients, the functions $f_a$ 
and $ S_{n,a}$ do not have zeros on $[- (a^2+1),0].$ 
For $x \in [0,a+1]$ we have 
$$ 1 \geq \frac{x}{a+1}  > \frac{x^2}{(a+1)(a^2+1)} >  \frac{x^3}{(a+1)(a^2+1)(a^3+1)}$$
$$  > \frac{x^4}{(a+1)(a^2+1)(a^3+1)(a^4+1)} > \cdots  , $$
whence 
\begin{equation}
\label{z1.1}
f_a(x) >0 \quad \mbox{for all}\quad x\in [0,a+1],
\end{equation}
and
\begin{equation}
\label{z1.2}
S_{n, a}(x) >0 \quad \mbox{for all}\quad x\in [0,a+1],\   n\geq 2.
\end{equation}

We have proved that  if $f_a \in \mathcal{L-P}$ (or 
$S_{n, a} \in \mathcal{L-P}$ for $n\geq 2$),   then 
there  exists $z_0 \in (a+1, a^2+1)$ such that $f_a(z_0) \leq 0$
(there  exists $z_n \in (a+1, a^2+1)$ such that $S_{n, a}(z_n) \leq 0$).

It remains to prove the inverse statement. To do that we need the following Lemma.

\begin{Lemma}
\label{th:lm3}
Denote by $\rho_j(f_a) := q_2(f_a)q_3(f_a) \ldots q_j(f_a) \sqrt{q_{j+1}(f_a)},$ 
$j\in\mathbb{N}.$ Then for every $j$ being large enough the function $f_a(z)$ has 
exactly $j$ zeros (counting multiplicities) in the circle $\{z:|z|<p_j(f_a)\}.$

\end{Lemma}
\begin{proof}

We use the denotations of $p_n$ and $q_n$ instead of $p_n(f_a)$ and $q_n(f_a)$. 
Then the function obtains the following form
$$f_a(z) = \sum_{k=0}^\infty \frac{(-1)^kz^k}{q_2^{k-1}q_3^{k-2} \ldots q_k},$$ 
where $q_2 < q_3 < \ldots$, $\lim\limits_{n \to \infty} q_k = a \geq 3.90155\ldots.$

We observe that 
$$f_a (z) = \sum_{k=0}^{j-3} \frac{(-1)^kz^k}{q_2^{k-1}q_3^{k-2} \ldots q_k} + 
\sum_{k = j-2}^{j+2}\frac{(-1)^kz^k}{q_2^{k-1}q_3^{k-2} \ldots q_k} + $$ 
$$ \sum_{k = j+3}^\infty \frac{(-1)^kz^k}{q_2^{k-1}q_3^{k-2} \ldots q_k}=: 
\Sigma_1(z) + g_j(z) + \Sigma_2(z).$$

We have
$$g_j(z) = \left( \sum_{k = j - 2}^{j+1} \frac{(-1)^kz^k}{q_2^{k-1} q_3^{k-2} \ldots q_k} + 
\frac{(-1)^{j+2}z^{j+2}}{q_2^{j+1}q_3^j \ldots q_{j-2}^5q_{j-1}^4q_j^4q_{j+1}^2}\right) +$$
$$\frac{(-1)^{j+2}z^{j+2}}{q_2^{j-3}q_3^{j-4} \ldots q_{j-2}} \bigg( \frac{1}{q_2^4q_3^4 
\ldots q_{j-1}^4q_j^3q_{j+1}^2q_{j+2}} - \frac{1}{q_2^4q_3^4 \ldots q_{j-1}^4q_j^4q_{j+1}^2}\bigg)$$ 
$$ =: \widetilde{g}_j(z) + \xi_j(z).$$

Let  $\rho_j := q_2q_3 \ldots q_j \sqrt{q_{j+1}}$. It follows that $q_2q_3 \ldots q_j  < 
\rho_j < q_2q_3 \ldots q_jq_{j+1}$.
We get
$$(-1)^{j-2}g_j(\rho_je^{i\theta}) = e^{i(j-2)\theta}q_2q_3^2 \ldots q_{j-2}^{j-3}q_{j-1}^{j-2}q_j^{j-2}
q_{j+1}^{\frac{j-2}{2}} \cdot \bigg( 1 - e^{i\theta}q_j \sqrt{q_{j+1}} $$
$$ + e^{2i\theta}q_jq_{j+1} - e^{3i\theta}q_j\sqrt{q_{j+1}} + e^{4i\theta}q_jq_{j+2}^{-1}\bigg)$$
$$= e^{i(j-2)\theta}q_2q_3^2 \ldots q_{j-2}^{j-3}q_{j-1}^{j-2}q_j^{j-2}q_{j+1}^{\frac{j-2}{2}} 
\cdot \bigg( 1 - e^{i\theta}q_j \sqrt{q_{j+1}}$$
$$ + e^{2i\theta}q_jq_{j+1} - e^{3i\theta}q_j\sqrt{q_{j+1}} + e^{4i\theta}\bigg)$$ 
$$+ e^{i(j+2)\theta}q_2q_3^2 \ldots q_{j-2}^{j-3}q_{j-1}^{j-2}q_j^{j-2}q_{j+1}^{\frac{j-2}{2}} 
\bigg(q_jq_{j+2}^{-1} - 1\bigg) =: \widetilde{g}_j(\rho_je^{i\theta}) + \xi_j(\rho_je^{i\theta}).$$

We show that for every sufficiently large $j$ the following inequality holds:  
$$\min_{0\leq \theta\leq 2\pi}|\widetilde{g}_j(\rho_j e^{i\theta})| > \max_{0\leq \theta\leq 2\pi}|f_a
(\rho_j e^{i\theta}) -\widetilde{g}_j(\rho_j e^{i\theta})|, $$ 
so that the number of zeros of  $f_a$  in the circle $\{z : |z| < \rho_j  \}$  is equal to  
the number of zeros of  $\widetilde{g}_j$ in the same circle. We have
 
$$\widetilde{g}_j(\rho_je^{i\theta}) =e^{ij\theta}q_2q_3^2 \ldots q_{j-2}^{j-3}q_{j-1}^{j-2}
q_j^{j-2}q_{j+1}^{\frac{j-2}{2}} 
\cdot \bigg( e^{-2i\theta} - e^{-i\theta}q_j \sqrt{q_{j+1}} + q_jq_{j+1} -$$
$$ e^{i\theta}q_j \sqrt{q_{j+1}}  + e^{2i\theta}\bigg)= e^{ij\theta}q_2q_3^2 \ldots 
q_{j-2}^{j-3}q_{j-1}^{j-2}q_j^{j-2}q_{j+1}^{\frac{j-2}{2}} \cdot \bigg( 2\cos 2\theta - 
2\cos\theta q_j \sqrt{q_{j+1}} $$
$$+ q_jq_{j+1} \bigg) =: e^{ij\theta}q_2q_3^2 \ldots q_{j-2}^{j-3}
q_{j-1}^{j-2}q_j^{j-2}q_{j+1}^{\frac{j-2}{2}} 
\cdot  \psi_j(\theta).$$

We find $\min_{0 \leq \theta \leq 2\pi}  |\widetilde{g}_j(\rho_je^{i\theta})|.$
Set $t:= \cos \theta$, $t \in [-1, 1]$. Then $\cos 2\theta = 2t^2 - 1,$ 
and
$$\psi_j(\theta) =\widetilde{\psi}_j(t) := 4t^2 - 2q_j\sqrt{q_{j+1}}t + (q_jq_{j+1} - 2).$$

The vertex of the parabola is $t_j = q_j\sqrt{q_j+1}/4.$ Under our assumption, $t_j > 1.$
Hence, 
$$\min_{t \in [-1, 1]} \widetilde{\psi}_j(t) = \widetilde{\psi}_j(1) =
2 - 2q_j\sqrt{q_{j+1}} + q_j q_{j+1} = q_j(\sqrt{q_{j+1}} -1)^2 - q_j + 2.$$

If $q_j \geq 4$, then $q_{j+1} \geq 4$, and 
$$q_j(\sqrt{q_{j+1}} -1)^2 - q_j + 2 \geq q_j - q_j + 2>0. $$
If $q_j <4$, then, since $q_j \geq q_2 \geq 3,$ we have $(\sqrt{q_{j+1}} - 1)^2 
\geq (\sqrt{3} -1)^2 > 0{.}5.$
Therefore,  we get
$$q_j(\sqrt{q_{j+1}} -1)^2 - q_j + 2 > q_j \frac{1}{2} - q_j = q_j\bigg(-\frac{1}{2}\bigg) + 2 > 0.$$
Thus, $\widetilde{\psi}_j(t) > 0$ for $t \in [-1, 1]$.

Consequently, we have obtained the estimation from below:
\begin{eqnarray}
\label{estg}
& \min_{0\leq \theta\leq 2\pi}|\widetilde{g}_j(\rho_j e^{i\theta})| \geq q_2q_3^2 \ldots 
q_{j-2}^{j-3}q_{j-1}^{j-2}q_j^{j-2}q_{j+1}^{\frac{j-2}{2}} \cdot 
\\ \nonumber &  \bigg(2 - 2q_j\sqrt{q_{j+1}} 
+ q_jq_{j+1} \bigg).
\end{eqnarray}

Now we will estimate the modulus of $\Sigma_1$ from above. We have

$$|\Sigma_1(\rho_j e^{i \theta})|\leq  \sum_{k = 0}^{j-3} \frac{q_2^kq_3^k \cdot\ldots 
\cdot q_j^k q_{j+1}^{\frac{k}{2}}}{q_2^{k-1}q_3^{k-2} \cdot \ldots \cdot q_k}= (\mbox{we rewrite 
the sum from right} $$
$$ \mbox{  to left})\ = \left( q_2 q_3^2 \cdot\ldots \cdot q_{j-3}^{j-4}q_{j-2}^{j-3}q_{j-1}^{j-3} q_j^{j-3} 
q_{j+1}^{\frac{j-3}{2}} + \right.$$
$$\left. q_2 q_3^2 \cdot\ldots \cdot q_{j-4}^{j-5} q_{j-3}^{j-4}q_{j-2}^{j-4} 
q_{j-1}^{j-4} q_j^{j-4} q_{j+1}^{\frac{j-4}{2}}  \right. +$$
 $$\left.   q_2 q_3^2 \cdot\ldots \cdot q_{j-5}^{j-6} q_{j-4}^{j-5}q_{j-3}^{j-5}q_{j-2}^{j-5} 
 q_{j-1}^{j-5} q_j^{j-5} q_{j+1}^{\frac{j-5}{2}} +
 \cdots    \right) = $$
 $$q_2 q_3^2 \cdot\ldots \cdot q_{j-3}^{j-4}q_{j-2}^{j-3}q_{j-1}^{j-3} 
 q_j^{j-3} q_{j+1}^{\frac{j-3}{2}} \cdot  \left( 1+ 
 \frac{1}{q_{j-2}q_{j-1}q_{j}\sqrt{q_{j+1}}}\right. $$
$$\left.    +  \frac{1}{q_{j-3}q_{j-2}^2 q_{j-1}^2q_{j}^2(\sqrt{q_{j+1}})^2} +
 \ldots\right) \leq   $$
 $$q_2 q_3^2 \cdot\ldots \cdot q_{j-3}^{j-4}q_{j-2}^{j-3}q_{j-1}^{j-3} q_j^{j-3} 
 q_{j+1}^{\frac{j-3}{2}}\cdot 
 \frac{1}{1-\frac{1}{q_{j-2}q_{j-1}q_{j}\sqrt{q_{j+1}}}} $$
 (we estimate the finite sum  from above by the sum of the infinite geometric 
 progression). Finally, we obtain
\begin{equation}
|\Sigma_1(\rho_j e^{i \theta})|\leq q_2 q_3^2 \cdot\ldots \cdot q_{j-3}^{j-4}q_{j-2}^{j-3} 
q_{j-1}^{j-3} q_j^{j-3} q_{j+1}^{\frac{j-3}{2}}\cdot 
 \frac{1}{1-\frac{1}{q_{j-2}q_{j-1}q_{j}\sqrt{q_{j+1}}}} .
\end{equation}

The estimation of $|\Sigma_2(\rho_j e^{i\theta})|$ from above can be made analogously.
$$|\Sigma_2(\rho_j e^{i \theta})|\leq  \sum_{k = j+3}^\infty \frac{q_2^kq_3^k \cdot\ldots 
\cdot q_j^k q_{j+1}^{\frac{k}{2}}}{q_2^{k-1}q_3^{k-2} \cdot \ldots \cdot q_k}=$$
$$ \frac{q_2q_3^2 \cdot \ldots \cdot q_j^{j-1} q_{j+1}^{\frac{j-3}{2}}}{q_{j+2}^2q_{j+3}}
\cdot \bigg( 1 + \frac{1}{\sqrt{q_{j+1}}q_{j+2}q_{j+3}q_{j+4}} $$
$$+ \frac{1}{(\sqrt{q_{j+1}})^2 
q_{j+2}^2q_{j+3}^2q_{j+4}^2q_{j+5}} + \ldots \bigg).$$

The latter can be estimated from above by the sum of the geometric progression, so, we obtain
\begin{equation}
|\Sigma_2(\rho_j e^{i \theta})|\leq \frac{q_2q_3^2 \cdot \ldots \cdot q_j^{j-1} 
q_{j+1}^{\frac{j-3}{2}}}{q_{j+2}^2 q_{j+3}}
\cdot \frac{1}{1 - \frac{1}{\sqrt{q_{j+1}}q_{j+2} q_{j+3} q_{j+4}}}.
\end{equation}

Note that

$|\xi_j(\rho_j e^{i \theta})| =  q_2q_3^2 \ldots q_{j-2}^{j-3}q_{j-1}^{j-2}q_j^{j-2}
q_{j+1}^{\frac{j-2}{2}} \bigg(q_jq_{j+2}^{-1} - 1\bigg).$

The desired inequality $\min_{0\leq \theta\leq 2\pi}|\widetilde{g}_j(\rho_j e^{i\theta})| > 
\max_{0\leq \theta\leq 2\pi}|f_a (\rho_j e^{i\theta}) -\widetilde{g}_j(\rho_j e^{i\theta})|$ 
follows from

$$q_2q_3^2 \ldots q_{j-2}^{j-3}q_{j-1}^{j-2}q_j^{j-2}q_{j+1}^{\frac{j-2}{2}} \cdot  
\bigg(2 - 2q_j\sqrt{q_{j+1}} + q_jq_{j+1} \bigg) > $$
$$q_2 q_3^2 \cdot\ldots \cdot q_{j-3}^{j-4}q_{j-2}^{j-3} 
q_{j-1}^{j-3} q_j^{j-3} q_{j+1}^{\frac{j-3}{2}}\cdot 
 \frac{1}{1-\frac{1}{q_{j-2}q_{j-1}q_{j}\sqrt{q_{j+1}}}} +$$
 $$\frac{q_2q_3^2 \cdot \ldots \cdot q_j^{j-1} 
q_{j+1}^{\frac{j-3}{2}}}{q_{j+2}^2 q_{j+3}}
\cdot \frac{1}{1 - \frac{1}{\sqrt{q_{j+1}}q_{j+2} q_{j+3} q_{j+4}}} + $$
$$q_2q_3^2 \ldots q_{j-2}^{j-3}q_{j-1}^{j-2}q_j^{j-2}q_{j+1}^{\frac{j-2}{2}} \bigg(q_jq_{j+2}^{-1} - 1\bigg).$$

Equivalently, 

\begin{eqnarray}
\label{estqq}
& q_{j-1}q_j\sqrt{q_{j+1}} \bigg(2 - 2q_j\sqrt{q_{j+1}} + q_jq_{j+1}\bigg) >
\\  \nonumber &
\frac{1}{1 - \frac{1}{q_{j-2}q_{j-1}q_j\sqrt{q_{j+1}}}} + 
\frac{q_{j-1}q_j^2}{q_{j+2}^2q_{j+3}} \frac{1}{1 - \frac{1}{\sqrt{q_{j+1}}q_{j+2}q_{j+3}q_{j+4}}} +
\\  \nonumber & q_{j-1}q_j\sqrt{q_{j+1}}(q_jq_{j+2}^{-1} - 1).
\end{eqnarray}

Since, under our assumptions, $\lim_{j \to \infty} q_j = a$, we investigate first the limiting
inequality
\begin{equation}
\label{esta} a^2\sqrt{a}(2 - 2a\sqrt{a} + a^2) > \frac{1}{1 - \frac{1}{a^3\sqrt{a}}} + 
\frac{1}{1 - \frac{1}{a^3\sqrt{a}}} + a^2\sqrt{a} \cdot 0.
\end{equation}

Equivalently, 
$$2 - 2a\sqrt{a} + a^2> \frac{2a}{a^3\sqrt{a}-1}.$$

Let $\sqrt{a} =: b$, then we obtain  $(2 - 2b^3 + b^4)(b^7 - 1) > 2b^2,$
or
$$b^{11} - 2b^{10} + 2b^7 - b^4 + 2b^3 - 2b^2 - 2 > 0.$$

The inequality is fulfilled for $b > 1{.}47$, thus,  for $a > 2{.}17.$
Under our assumptions, $a > 3{.}57,$ so the inequality (\ref{esta})
is valid for our assumptions on $a.$ Whence, the inequality (\ref{estqq})
is valid for our assumptions on $a$ and for all $j$ being large enough.

Consequently, we have proved that for all $j$ being large enough 
$\min_{0\leq \theta\leq 2\pi}|\widetilde{g}_j(\rho_j e^{i\theta})| > 
\max_{0\leq \theta\leq 2\pi} |f_a(\rho_j e^{i\theta}) -\widetilde{g}_j(\rho_j e^{i\theta})|, $ 
so the numbers of zeros of 
$f_a$ in the circle $\{z:|z| < \rho_j\}$ is equal to the numbers of zeros of 
$\widetilde{g}_j$ in this circle.

It remains to find the number of zeros of $\widetilde{g}_j$ in the circle $\{z:|z| < \rho_j\}$. We have 

$$\widetilde{g}_j(z) = \sum_{k = j - 2}^{j+1} \frac{(-1)^kz^k}{q_2^{k-1} q_3^{k-2} \ldots q_k} + 
\frac{(-1)^{j+2}z^{j+2}}{q_2^{j+1}q_3^j \ldots q_{j-2}^5q_{j-1}^4q_j^4q_{j+1}^2}.$$

Let us use the denotation $w = z \rho_j^{-1},$ so that $|w|<1.$ This yields

$$\widetilde{g}_j(\rho_j w) = (-1)^{j-2}w^{j-2}q_2q_3^2 \ldots 
q_{j-2}^{j-3}q_{j-1}^{j-2}q_j^{j-2}q_{j+1}^{\frac{j-2}{2}} \cdot $$
$$(1 - q_j\sqrt{q_{j+1}}w + q_jq_{j+1}w^2 - q_j\sqrt{q_{j+1}}w^3 + w^4).$$

It follows from (\ref{estg}) that $\widetilde{g}_j$ does not have zeros on the circumference $\{z : |z| = \rho_j  \},$
whence $\widetilde{g}_j(\rho_j w)$ does not have zeros on the circumference $\{w : |w| = 1  \}.$ Since 
$P_j(w) =1-q_j\sqrt{q_{j+1}} w +q_j q_{j+1} w^2 - q_j\sqrt{q_{j+1}} w^3 +w^4 $  is a 
self-reciprocal polynomial on $w,$ we can conclude that $P_j$ has exactly two zeros in the circle $\{ w :|w| <1 \}.$
Hence, $\widetilde{g}_j(z)$ has exactly $j$ zeros in the circle $\{z : |z| < \rho_j  \},$ and we have 
proved the statement of  Lemma~\ref{th:lm3}.
\end{proof}

\begin{Lemma}
\label{th:lm6} Denote by $\rho_k(f_a) := q_2(f_a)q_3(f_a) \ldots q_k(f_a) \sqrt{q_{k+1}(f_a)},$ 
$k\in\mathbb{N}.$
If $a\geq 3$ then for every $k\geq 2$  the following inequality holds: $$(-1)^k f_a(\rho_k) \geq 0.$$
\end{Lemma}

\begin{proof}

We use the denotations of $p_n,$  $q_n$ and $\rho_n$  instead of $p_n(f_a),$  $q_n(f_a)$ 
and $\rho_n(f_a).$ 
Then the function obtains the following form
$$f_a(z) = \sum_{j=0}^\infty \frac{(-1)^j z^j}{q_2^{j-1}q_3^{j-2} \ldots q_j},$$ 
where $q_2 < q_3 < \ldots$, $\lim\limits_{k \to \infty} q_k = a \geq 3.$

Since  $ \rho_k  \in (q_2q_3 \ldots q_k, q_2q_3 \ldots q_kq_{k+1}),$ we have 
$$ 1< \rho_k < \frac{\rho_k^2}{q_2} < \cdots < \frac{\rho_k^k}{q_2^{k-1}q_3^{k-2} \ldots q_k},  $$
and
$$ \frac{\rho_k^k}{q_2^{k-1}q_3^{k-2} \ldots q_k} > \frac{\rho_k^{k+1}}{q_2^{k}q_3^{k-1} \ldots q_k^2 q_{k+1}} 
> \frac{\rho_k^{k+2}}{q_2^{k+1}q_3^{k} \ldots q_k^3 q_{k+1}^2 q_{k+2}} > \cdots  .$$
Therefore, we get for $k\geq 2$
$$(-1)^kf_a(\rho_k) \geq \sum_{j=k-3}^{k+3}\frac{(-1)^{j+k}\rho_k^j}{q_2^{j-1} q_3^{j-2} \ldots q_j} =: \mu_k(\rho_k),$$
and it is sufficient to prove that for every $k\geq 2$ we have $\mu_k(\rho_k)\geq 0.$  
After reducing by $\frac{\rho_k^{k-3}}{q_2^{k-4} q_3^{k-5} \ldots q_{k-3}}$ we get 
the inequality we desired:
$$ -1 + \frac{\rho_k}{q_2q_3 \ldots q_{k-3}q_{k-2}} - \frac{\rho_k^2}{q_2^2q_3^2 
\ldots q_{k-2}^2q_{k-1}} + \frac{\rho_k^3}{q_2^3q_3^3 \ldots q_{k-2}^3q_{k-1}^2q_k} -$$
$$\frac{\rho_k^4}{q_2^4q_3^4 \ldots q_{k-2}^4q_{k-1}^3q_k^2q_{k+1}} + 
\frac{\rho_k^5}{q_2^5q_3^5 \ldots q_{k-2}^5q_{k-1}^4q_k^3q_{k+1}^2q_{k+2}} -$$
$$\frac{\rho_k^6}{q_2^6q_3^6 \ldots q_{k-2}^6q_{k-1}^5q_k^4q_{k+1}^3q_{k+2}^2q_{k+3}} 
\geq 0,$$
or, using that $\rho_k= q_2q_3 \ldots q_k \sqrt{q_{k+1}},$
$$\nu_k(\rho_k) := -1 + q_{k-1}q_k\sqrt{q_{k+1}} - 2q_{k-1}q_k^2q_{k+1} + 
q_{k-1}q_k^2q_{k+1}\sqrt{q_{k+1}}+$$
$$q_{k-1}q_k^2\sqrt{q_{k+1}}q_{k+2}^{-1} - 
q_{k-1}q_k^2q_{k+2}^{-2}q_{k+3}^{-1} \geq 0.$$

We rewrite and get 
$$\nu_k(\rho_k) = q_{k-1}q_k^2q_{k+1}\sqrt{q_{k+1}} - 2q_{k-1}q_k^2q_{k+1}+$$
$$q_{k-1}q_k\sqrt{q_{k+1}} \bigg( 1 + \frac{q_k}{q_{k+2}}\bigg) - \bigg(1 + 
\frac{q_{k-1}q_k^2}{q_{k+2}^2q_{k+3}}\bigg) \geq 0.$$

It is easy to check that  $\frac{q_k}{q_{k+2}}$ increases in $k$ to $1$, so we have 
$\frac{q_k}{q_{k+2}} \geq \frac{q_2}{q_4} = \frac{a^5 + a^3 + a^2 + 1}{a^5 + a^4 + a + 1} \geq 0{.}8$ 
for $a \geq 0.$

In addition, $\frac{q_{k-1}q_k^2}{q_{k+2}^2q_{k+3}} < 1$, so it is 
sufficient to prove the following inequality
$$q_{k-1}q_k^2q_{k+1}\sqrt{q_{k+1}} - 2q_{k-1}q_k^2q_{k+1}+
q_{k-1}q_k\sqrt{q_{k+1}} \cdot 1{.}8 - 2 \geq 0.$$

Since $2 < \frac{2}{9}q_{k-1}q_k$, we can observe that

$$q_{k-1}q_k^2q_{k+1}\sqrt{q_{k+1}} - 2q_{k-1}q_k^2q_{k+1}+
q_{k-1}q_k\sqrt{q_{k+1}} \cdot 1{.}8 - 2 \geq$$
$$q_{k-1}q_k^2q_{k+1}\sqrt{q_{k+1}} - 2q_{k-1}q_k^2q_{k+1}+
q_{k-1}q_k\sqrt{q_{k+1}} \cdot 1{.}8 - $$
$$\frac{2}{9}q_{k-1}q_k .$$

So, we need to check that for all $k\geq 2$ 
$$q_kq_{k+1}\sqrt{q_{k+1}} - 2q_kq_{k+1} + 1.8\sqrt{q_{k+1}} - \frac{2}{9}  =
q_kq_{k+1}\left(\sqrt{q_{k+1}} - 2\right) + 1{.}8\sqrt{q_{k+1}} - \frac{2}{9} \geq 0.$$

If $q_{k+1} \geq 4,$ then $\sqrt{q_{k+1}} - 2\ \geq 0$ and $1.8\sqrt{q_{k+1}} - 
\frac{2}{9} \geq 0,$ and the last inequality is valid. If $q_{k+1} < 4,$
since $q_k$ increases in $k$, we get
$$q_kq_{k+1}\sqrt{q_{k+1}} - 2q_kq_{k+1} + 1.8\sqrt{q_{k+1}} - \frac{2}{9} \geq$$
$$q_{k+1}^2\sqrt{q_{k+1}} - 2q_{k+1}^2 + 1{.}8\sqrt{q_k+1} - \frac{2}{9} \geq 0.$$

Let $\sqrt{q_{k+1}} = t, t \geq 0$, we obtain the following inequality
$$t^5 - 2t^4 + 1.8t - \frac{2}{9} \geq 0.$$
This inequality holds for $t \geq 1.57685 \ldots,$ so it follows that it holds
for $q_{k+1} \geq 2.48646\ldots$. 

Lemma~\ref{th:lm6} is proved. 
\end{proof}

Suppose that there  exists $z_0 \in (a+1, a^2+1),$ such that $f_a(z_0) \leq 0.$ Then,
by Lemma~\ref{th:lm6}  we have  for every $k\geq 2:$
$$  f_a(0) >0,  f_a(z_0) \leq 0, f_a(\rho_2) \geq 0, f_a(\rho_3) \leq 0, \ldots, 
(-1)^k f_a(\rho_k) \geq 0. $$
So, for every $k\geq 2$ the function $f_a$ has at least $k-1 $ real zeros in the circle 
$\{ z : |z| < \rho_k   \}.$ By Lemma~\ref{th:lm3} the function $f_a$  has exactly
$k$ zeros in the circle $\{ z : |z| < \rho_k   \}$ for $k$ being large enough. Thus, 
if there  exists $z_0 \in (a+1, a^2+1),$ such that $f_a(z_0) \leq 0,$ then all the
zeros of $f_a$ are real.

Theorem~\ref{th:mthm1} is proved.

\section{Proof of Theorem \ref{th:mthm2}}

In order to estimate from below the values of $a$  such that $f_a$  belongs to the 
Laguerre-P\'olya class, we consider its section $S_{3, a}(z) = 1 - \frac{z}{a+1} + \frac{z^2}{(a+1)(a^2+1)} + 
\frac{z^3}{(a+1)(a^2+1)(a^3+1)}.$

We have proved that if $f_a \in \mathcal{L-P}$, then there exists $z_0\in (a+1, a^2+1)$ such that 
$f_a(z_0) \leq 0.$ Note that, for every $z \in (a+1, a^2+1)$  we have $S_{3,a}(z) < f_a(z),$ whence
 $S_{3,a}(z_0) \leq 0.$

\begin{Lemma}
\label{th:lm4}
If there exists $z_0 \in (a+1, a^2+1)$ such that $S_{3,a}(z_0) \leq 0,$ then $a \geq 3{.}90155.$ 
\end{Lemma}

\begin{proof} We  denote by $y_0:= \frac{z_0}{a+1}$, and get
$1 < y_0 < q_2.$ Hence we get 
$$S_{3,a}(z_0) =S_{3,a}((a+1)y_0) = 1 - y_0 + \frac{y_0^2}{q_2} - \frac{y_0^3}{q_2^2q_3}, 3 \leq q_2 < q_3.$$

Let us use the denotation $b:=q_2, c:= q_3$. Then we obtain
$$S_{3,a}((a+1)y_0) = 1 - y_0 + \frac{y_0^2}{b} - \frac{y_0^3}{b^2c}=: K(y_0).$$

First, we find the roots of the derivative. The derivative of $K(y) = 1 - y + \frac{y^2}{b} - \frac{y^3}{b^2c}$ is
$$K\textprime(y) = -\frac{1}{b^2c}(3y^2 - 2bc y + b^2c).$$

We consider the discriminant of the quadratic polynomial $K\textprime$: 
$D/4 = b^2c^2 - 3b^2c = b^2c (c - 3) >0$, since under our assumptions, $c > 3.$
Thus, the roots are 
$$y_1 = \frac{bc - b \sqrt{c(c-3)}}{3}$$
and
$$y_2 = \frac{bc + b \sqrt{c(c-3)}}{3}.$$

We want to check if $y_1$ or $y_2$ lie in $(1, b).$

\begin{equation}
\label{bc}
1 < \frac{bc - b \sqrt{c(c-3)}}{3} < b.
\end{equation}
The left-hand side of the inequality is 
$bc - 3 > b\sqrt{c(c-3)}$ or
$b^2c^2 - 6bc + 9 >  b^2c(c-3).$
The inequality is fulfilled under our assumptions when $b<c$: 
$b^2c - 2bc + 3 = bc(b-2) + 3 > 0.$

Now we consider the right-hand side of (\ref{bc}). It is equivalent to $bc - b \sqrt{c(c-3)} < 3b,$
or $c^2 - 6c + 9 < c^2 - 3c.$ Under our assumptions, $c - 3 > 0,$ so the inequality is fulfilled.

Thus, we have verified that $y_1 \in (1, b).$

Now we check that $y_2 > b,$ or
$$ \frac{bc + b \sqrt{c(c-3)}}{3} > b.$$
Equivalently, $c +  \sqrt{c(c-3)} > 3,$ which is true by our assumptions for $c.$

So we get that $y_1$ is the minimal point of $K(y)$ in the interval $1 < y < q_2.$
Thus, there exists $y_0,  1 < y_0 < q_2,$ such that $K(y_0) \leq 0,$ implies $K(y_1) \leq 0.$

After substituting $y_1$ into $K(y)$, we obtain the following expression
$$K(y_1) = 1 - \frac{bc - b\sqrt{c(c-3)}}{3} + \frac{(bc - b\sqrt{c(c-3)})^2}{9b} - 
\frac{(bc - b\sqrt{c(c-3)})^3}{27b^2c}.$$
We want $K(y_1) \leq 0,$  or 
$$27 - 9bc + 9b\sqrt{c(c-3)} + 3bc^2 - 6bc\sqrt{c(c-3)} + 3bc(c-3) - bc^2+$$
$$3bc\sqrt{c(c-3)} - 3bc(c-3) + b(c-3)\sqrt{c(c-3)} \leq 0.$$

We rewrite and get
\begin{equation}
\label{bc} \sqrt{c(c-3)}(6b - 2bc) + (27 - 9bc + 2bc^2) \leq 0.
\end{equation}

We have $6b - 2bc = 2b(3 - c) < 0$, since $c > 3$. Thus, $\sqrt{c(c - 3)}(6b - 2bc) < 0.$

Now we show that the following inequality is fulfilled
$$27 - 9bc + 2bc^2 \geq 0.$$

We substitute $b = q_2 = \frac{a^2 + 1}{a +1}, c = q_3 = \frac{a^3 + 1}{a^2 + 1}.$
We have
$$27 - 9\frac{a^2 + 1}{a +1} \cdot \frac{a^3 + 1}{a^2 + 1} + 2\frac{a^2 + 1}{a +1} \cdot 
\left(\frac{a^3 + 1}{a^2 + 1}\right)^2 \geq 0.$$

Equivalently,
$$2a^6 - 9a^5 + 22a^3 + 18a^2 + 27a + 20 \geq 0,$$
or
$$(a+1)^2(2a^4 - 13a^3 + 24a^2 - 13a + 20) \geq 0.$$
It remains to prove that
$$2a^4 - 13a^3 + 24a^2 - 13a + 20 \geq 0.$$

Since $a \geq 3$, let $a = 3+x, x \geq 0.$ We obtain
$$2(3+x)^4 - 13(3+x)^3 + 24(3+x)^2 - 13(3+x) + 20 \geq 0,$$
or
$$2x^4 + 11x^3 + 15x^2 - 4x + 8 \geq 0, $$
which is true for $x \geq 0.$

Consequently, the inequality $$27 - 9bc + 2bc^2 \geq 0$$ is verified.

Thus, we rewrite (\ref{bc}) in the form
$$27 - 9bc + 2bc^2 \leq \sqrt{c(c-3)}(2bc - 6b).$$

We can observe that both sides of the inequality are positive. After straightforward calculations we get
$$b^2c^2 - 4b^2c + 18bc - 4bc^2 - 27 \geq 0.$$

We substitute $b = q_2 = \frac{a^2 + 1}{a +1}, c = q_3 = \frac{a^3 + 1}{a^2 + 1}$ and obtain

$$\frac{(a^3+1)^2}{(a+1)^2} - 4\frac{(a^3+1)(a^2+1)}{(a+1)^2} + 
18\frac{a^3+1}{a+1} - 4 \frac{(a^3+1)^2}{(a+1)(a^2+1)} - 27 \geq 0,$$
Or, equivalently, 
$$a^8 - 8a^7 + 15a^6 + 12a^5 - 21a^4 - 28a^3 - 43a^2 - 40a - 16 \geq 0.$$
The last inequality is valid when
$$a \geq 3{.}90155\ldots. $$
Lemma~\ref{th:lm4} is proved. 
\end{proof}

For any $z_0 \in (1, q_2)$ and for any $n \in \mathbb{N}$: $S_{2n+1,a}(z_0) \leq f_a(z_0) 
\leq S_{2n,a}(z_0).$ So  if there exists $z_0\in (a+1, a^2+1)$ such that $S_{6,a}(z_0) \leq 0,$
then    $f_a(z_0) \leq 0.$

\begin{Lemma}
\label{th:lm5}
If $a \geq  3.91719,$ then there exists $z_0 \in (a+1, a^2+1)$ such that $S_{6,a}(z_0) \leq 0.$ 

\end{Lemma}

\begin{proof}  We choose $z_0 = \frac{2}{3}(a+1)q_2 \in (a+1, a^2+1).$ Then we get

$$S_{6,a}\bigg(\frac{2}{3}(a+1)q_2 \bigg) = 1 - \frac{2}{9}q_2 - \frac{8}{27}\frac{q_2}{q_3} + \frac{16}{81}\frac{q_2}{q_3^2q_4} -
\frac{32}{243}\frac{q_2}{q_3^3q_4^2q_5} + \frac{64}{729}\frac{q_2}{q_3^4q_4^3q_5^2q_6}.$$
We need the inequality $S_{6,a}(z_0) \leq 0$ to be fulfilled.

Now we rewrite the inequality using $q_j = \frac{a^j+1}{a^{j-1}+1}$ and after direct calculations we obtain the following:
$$729(a+1)(a^3+1)(a^4+1)(a^5+1)(a^6+1) - 162 (a^2+1)(a^3+1)(a^4+1)(a^5+1) \cdot $$
$$(a^6+1) - 216(a^2+1)^2(a^4+1)(a^5+1)(a^6+1) + 144(a^2+1)^3(a^5+1)(a^6+1) - $$
$$96(a^2+1)^4(a^5+1) + 64(a^2+1)^5 \leq 0.$$

Equivalently, 
$$-162a^{20} + 513a^{19} + 567a^{18} - 594a^{17} + 567a^{16} + 1134a^{15} + 918a^{14} + 822a^{13} +$$ 
$$846a^{12} + 228a^{11}+ 1927a^{10} + 1125a^9 + 1142a^8 + 750a^7 + 1030a^6 + 966a^5 +$$ 
$$1360a^4+567a^3-226a^2+729a+463 \leq 0.$$

The inequality above is valid for $a \geq 3.91719.$

Lemma~\ref{th:lm5} is proved. 
\end{proof}

Theorem \ref{th:mthm2} is proved.

{\bf Aknowledgement.} This article was partially supported by the Akhiezer Foundation. The author is also deeply grateful to her scientific advisor, A. Vishnyakova, for her patience, support and contribution to her research career.

\end{document}